\documentclass{amsart}

\usepackage[section]{placeins}

\usepackage{amsmath}             
\usepackage{amscd}
\usepackage{amssymb}             
\usepackage{amsfonts}            
\usepackage{latexsym}            
\usepackage{amsthm}              
\usepackage{graphicx}

\usepackage{enumerate}
\usepackage{mathrsfs} 
\usepackage{thmtools}
\usepackage{thm-restate}

\newcommand{\SL}{\operatorname{SL}}

\newtheorem{lause}{Theorem}[section]

\newtheorem{lemma}[lause]{Lemma}
\newtheorem{seur}[lause]{Corollary}
\newtheorem{prop}[lause]{Proposition}

\newtheorem*{lause*}{Theorem}

\theoremstyle{definition}
\newtheorem{maar}[lause]{Definition}

\theoremstyle{remark}
\newtheorem{remark}[lause]{Remark}
 
\newtheorem*{mot*}{Motivation}
\newtheorem*{acknow*}{Acknowledgements}

\numberwithin{equation}{section}

\usepackage{hyphenat}
\allowdisplaybreaks

\begin{document}

\title[Exterior and symmetric squares in characteristic two]{Decomposition of exterior and symmetric squares in characteristic two}

\author{\vspace{-2ex}Mikko Korhonen}
\address{Department of Mathematics, Southern University of Science and Technology, \text{Shenzhen} 518055, Guangdong, P. R. China}
\email{korhonen\_mikko@hotmail.com}
\thanks{Partially supported by NSFC grants 11771200 and 11931005.}

\date{\today}

\begin{abstract}

Let $V$ be a finite-dimensional vector space over a field of characteristic two. As the main result of this paper, for every nilpotent element $e \in \mathfrak{sl}(V)$, we describe the Jordan normal form of $e$ on the $\mathfrak{sl}(V)$-modules $\wedge^2(V)$ and $S^2(V)$. In the case where $e$ is a regular nilpotent element, we are able to give a closed formula.

We also consider the closely related problem of describing, for every unipotent element $u \in \operatorname{SL}(V)$, the Jordan normal form of $u$ on $\wedge^2(V)$ and $S^2(V)$. A recursive formula for the Jordan block sizes of $u$ on $\wedge^2(V)$ was given by Gow and Laffey (J. Group Theory 9 (2006), 659--672). We show that their proof can be adapted to give a similar formula for the Jordan block sizes of $u$ on $S^2(V)$.

\end{abstract}

\maketitle

\section{Introduction}

Let $V$ be a finite-dimensional vector space over a field. Let $u \in \SL(V)$ be a unipotent linear map and let $e \in \mathfrak{sl}(V)$ be a nilpotent linear map. We consider the following two basic questions in representation theory. 

\begin{itemize}
\item[\textbf{Q1.}] What are the Jordan block sizes of $u$ in its $\SL(V)$-action on the exterior square $\wedge^2(V)$ and the symmetric square $S^2(V)$? 
\item[\textbf{Q2.}] What are the Jordan block sizes of $e$ in its $\mathfrak{sl}(V)$-action on $\wedge^2(V)$ and $S^2(V)$? 
\end{itemize}

As we will see later in this introduction, results from the literature quickly reduce both problems to the case where $u$ and $e$ act on $V$ with a single Jordan block, so we assume that this is the case. Then good answers to both questions are known in odd characteristic \cite[Theorem 2]{Barry} \cite[Theorem 24]{McNinchAdjoint}. 

In this paper we will consider the characteristic two case, where the previously known results are as follows. A formula for the Jordan block sizes of $u$ on $\wedge^2(V)$ has been given by Gow and Laffey \cite[Theorem 2]{GowLaffey}. With \cite[Theorem 2]{GowLaffey} and \cite[Corollary 3.11]{SymondsSymtoExt}, one can calculate the Jordan decomposition of $u$ on $S^2(V)$, modulo Jordan blocks of even size. Then \cite[Proposition 2.2]{SymondsSymtoExt} provides a recursive algorithm for computing the Jordan block sizes of $u$ on $S^2(V)$.

The main purpose of this paper is to provide explicit formulae in characteristic two for the Jordan block sizes of $u$ on $S^2(V)$ (Theorem \ref{thm:unipsym}) and the Jordan block sizes of $e$ on $\wedge^2(V)$ and $S^2(V)$ (Theorems \ref{thm:mainthmextnilpotent} -- \ref{thm:mainthmrecursivenilpotentsym}). 

For the Jordan block sizes of $u$ on $S^2(V)$, we give a recursive formula which is analogous to \cite[Theorem 2]{GowLaffey}. In the nilpotent case, we will compute a Jordan basis for the action of $e$ on $V \otimes V$ (Theorem \ref{thm:basistensorsquare}), and use it to find a closed formula for the Jordan block sizes of $e$ on $\wedge^2(V)$ and $S^2(V)$ (Theorems \ref{thm:mainthmextnilpotent} -- \ref{thm:mainthmsymnilpotent}).

For the rest of this paper, we fix a field $K$ and make the following assumption. 

\begin{center}\emph{Assume that $\operatorname{char} K = 2$.}\end{center} 

To describe our results, it will be convenient to do so in terms of representation theory. Let $q = 2^{\alpha}$, where $\alpha$ is a positive integer. Let $C_q$ be a cyclic group of order $q$. Recall that there are a total of $q$ indecomposable $K[C_q]$-modules $V_1$, $\ldots$, $V_q$, where $\dim V_i = i$ and a generator of $C_q$ acts on $V_i$ as a single $i \times i$ unipotent Jordan block. Denote $V_0 = 0$. For a $K$-vector space $W$, we denote $W^0 = 0$ and $W^d = W \oplus \cdots \oplus W$ ($d$ copies) for an integer $d > 0$.

Then question \textbf{Q1} is equivalent to the problem of decomposing $\wedge^2(V)$ and $S^2(V)$ into indecomposable summands for every $K[C_q]$-module $V$. Note that we have isomorphisms \begin{align}\label{eq:additivext}\wedge^2(V \oplus W) &\cong \wedge^2(V) \oplus (V \otimes W) \oplus \wedge^2(W), \\ \label{eq:additivesym}S^2(V \oplus W) &\cong S^2(V) \oplus (V \otimes W) \oplus S^2(W)\end{align} of $K[C_q]$-modules. Thus \textbf{Q1} is reduced to the problem of decomposing $V_m \otimes V_n$, $\wedge^2(V_n)$, and $S^2(V_n)$ into indecomposable summands for integers $0 < n,m \leq q$.

The decomposition of $V_m \otimes V_n$ has been extensively studied in all characteristics, see for example \cite{Srinivasan}, \cite{Ralley}, \cite{McFall}, \cite{Renaud}, \cite{Norman}, \cite{Norman2}, \cite{Hou}, and \cite{Barry}. In our setting of characteristic two, we can use the following result, which gives a recursive description for the decomposition of $V_m \otimes V_n$.

\begin{lause}[{\cite[(2.5a)]{GreenModular}, \cite[Lemma 1, Corollary 3]{GowLaffey}}]\label{thm:tensordecompchar2}
Let $0 < m \leq n \leq q$ and suppose that $q/2 < n \leq q$. Then the following statements hold:
\begin{enumerate}[\normalfont (i)]
\item If $n = q$, then $V_m \otimes V_n \cong V_q^m$ as $K[C_q]$-modules.
\item If $m+n > q$, then $V_m \otimes V_n \cong V_q^{n+m-q} \oplus (V_{q-n} \otimes V_{q-m})$ as $K[C_q]$-modules.
\item If $m+n \leq q$, then $V_m \otimes V_n \cong V_{q-d_t} \oplus \cdots \oplus V_{q-d_1}$ as $K[C_q]$-modules, where $V_m \otimes V_{q-n} \cong V_{d_1} \oplus \cdots \oplus V_{d_t}$.
\end{enumerate}
\end{lause}

Note that with Theorem \ref{thm:tensordecompchar2}, we are able to calculate $V_m \otimes V_n$ for any given $0 < m \leq n \leq q$. We either get an explicit decomposition (case (i)), or an expression of $V_m \otimes V_n$ in terms of a tensor product $V_{m'} \otimes V_{n'}$ for some $0 < m' \leq n'$ with $n' < n$. In the latter case we can consider $V_{m'} \otimes V_{n'}$ as a $K[C_{q'}]$-module, where $q'$ is a power of $2$ such that $q'/2 < n' \leq q'$. Thus by applying Theorem \ref{thm:tensordecompchar2} repeatedly, we can quickly calculate the decomposition of $V_m \otimes V_n$ into indecomposable summands.

For the decomposition of $\wedge^2(V_n)$, a recursive formula in similar vein as Theorem \ref{thm:tensordecompchar2} was found by Gow and Laffey \cite{GowLaffey}.

\begin{lause}[{\cite[Theorem 2]{GowLaffey}}]\label{thm:unipext}
Suppose that $q/2 < n \leq q$. Then we have $$\wedge^2(V_n) \cong \wedge^2(V_{q-n}) \oplus V_q^{n-q/2-1} \oplus V_{3q/2-n}$$ as $K[C_q]$-modules.
\end{lause}

It turns out that there is a similar recurrence for the decomposition of $S^2(V_n)$, which we will prove in the next section. Our proof will follow along the same lines as the proof of Theorem \ref{thm:unipext} in \cite{GowLaffey}.

\begin{lause}\label{thm:unipsym}
Suppose that $q/2 < n \leq q$. Then we have $$S^2(V_n) \cong \wedge^2(V_{q-n}) \oplus V_q^{n-q/2} \oplus V_{q/2}$$ as $K[C_q]$-modules.
\end{lause}

Note that in Theorem \ref{thm:unipext} we have $q-n < q/2$, so the result can be applied repeatedly to find efficiently the decomposition of $\wedge^2(V_n)$ for any given $n$. Similarly $S^2(V_n)$ can be decomposed by applying Theorem \ref{thm:unipsym} together with Theorem \ref{thm:unipext}. 

The main part of this paper will be concerned with problem \textbf{Q2} about nilpotent linear maps. Here the most natural way to describe our results will be in terms of representations of Lie algebras. Let $\mathfrak{w}_q$ be the abelian $p$-Lie algebra over $K$ generated by a single nilpotent element $e \in \mathfrak{w}_q$ such that $e^{[q]} = 0$, so as a $K$-vector space $$\mathfrak{w}_q = \bigoplus_{0 \leq i < \alpha} \langle e^{[2^i]} \rangle.$$ There are a total of $q$ indecomposable restricted $\mathfrak{w}_q$-modules $W_1$, $\ldots$, $W_q$, where $\dim W_i = i$ and $e$ acts on $W_i$ as a single $i \times i$ nilpotent Jordan block. Denote $W_0 = 0$. In analogue with the unipotent case, question \textbf{Q2} is equivalent to the problem of decomposing $\wedge^2(V)$ and $S^2(V)$ into indecomposable summands for every restricted $\mathfrak{w}_q$-module $V$. 

The isomorphisms~\eqref{eq:additivext} and~\eqref{eq:additivesym} hold for $\mathfrak{w}_q$-modules as well, so we are reduced the problem of decomposing $W_m \otimes W_n$, $\wedge^2(W_n)$, and $S^2(W_n)$ into indecomposable summands. 
 
By the following result, one can calculate the decomposition of $W_m \otimes W_n$ using Theorem \ref{thm:tensordecompchar2}. This is a special case of a result of Fossum \cite{Fossum} on formal group laws, alternatively a short proof can be found in \cite[Corollary 5 (a)]{NormanTwoRelated}.

\begin{prop}[{\cite[Section III]{Fossum}}]\label{prop:uninilconnection}
Let $0 < n,m \leq q$ and suppose that we have $V_m \otimes V_n \cong V_{r_1} \oplus \cdots \oplus V_{r_t}$ as $K[C_q]$-modules for some $r_1, \ldots, r_t > 0$. Then $W_m \otimes W_n \cong W_{r_1} \oplus \cdots \oplus W_{r_t}$ as $\mathfrak{w}_q$-modules.
\end{prop}

The analogue of Proposition \ref{prop:uninilconnection} fails for $\wedge^2(W_n)$ and $S^2(W_n)$. The following example was noted in \cite[p. 286]{Fossum}: we have $\wedge^2(V_4) \cong V_2 \oplus V_4$, but $\wedge^2(W_4) \cong W_3^2$. Furthermore, we have $S^2(V_3) \cong V_2 \oplus V_4$, but $S^2(W_3) \cong W_1^2 \oplus W_4$.

In our main results for $\wedge^2(W_n)$ and $S^2(W_n)$, we will give a closed formula for their decomposition into indecomposable summands. For this, we will need the following definition from \cite[p. 231]{GlasbyPraegerXiapart}.

\begin{maar}\label{def:consecutiveones} The \emph{consecutive-ones binary expansion} of an integer $n > 0$ is the alternating sum $n = \sum_{1 \leq i \leq r} (-1)^{i+1} 2^{\beta_i}$ such that $\beta_1 > \cdots > \beta_r \geq 0$ and $r$ is minimal.\end{maar}

For example, we have consecutive-ones binary expansions $3 = 2^2 - 2^0$, $4 = 2^2$, $5 = 2^3 - 2^2 + 2^0$, $6 = 2^3 - 2^1$, and $7 = 2^3 - 2^0$. Note that for any consecutive-ones binary expansion, we have $\beta_{r-1} > \beta_r + 1$ if $r > 1$.

Using the consecutive-ones binary expansion of $n$, Glasby, Praeger and Xia have given an explicit expression for the indecomposable summands of $V_n \otimes V_n$ and their multiplicities \cite[Theorem 15]{GlasbyPraegerXiapart}. By Proposition \ref{prop:uninilconnection}, this also gives us the decomposition of $W_n \otimes W_n$ into indecomposable summands. 

In Section \ref{section:jordanbasis}, we give a different proof of \cite[Theorem 15]{GlasbyPraegerXiapart} by constructing a Jordan basis for the action of $e$ on $W_n \otimes W_n$. This Jordan basis can be used to find Jordan bases for the action of $e$ on $\wedge^2(W_n)$ and $S^2(W_n)$ as well, allowing us to compute the indecomposable summands of $\wedge^2(W_n)$ and $S^2(W_n)$ explicitly. This leads to the following results, which will be proven in Section \ref{section:jordanbasis2}.

\begin{lause}\label{thm:mainthmextnilpotent}
Let $n > 0$ be an integer, with consecutive-ones binary expansion $n = \sum_{1 \leq i \leq r} (-1)^{i+1} 2^{\beta_i}$, where $\beta_1 > \cdots > \beta_r \geq 0$. For $1 \leq k \leq r$ with $\beta_k > 0$, define $d_k := 2^{\beta_k - 1} + \sum_{k < i \leq r} (-1)^{k+i} 2^{\beta_i}$. Then $$\wedge^2(W_n) \cong \bigoplus_{\substack{1 \leq k \leq r \\ \beta_k > 0}} W_{2^{\beta_k}-1}^{d_k}$$ as $\mathfrak{w}_q$-modules.\end{lause}

\begin{lause}\label{thm:mainthmsymnilpotent}
Let $n > 0$ be an integer, and let $\beta_1 > \cdots > \beta_r \geq 0$ and $d_k$ be as in Theorem \ref{thm:mainthmextnilpotent}. Then $$S^2(W_n) \cong W_1^{\lceil n/2 \rceil} \oplus \bigoplus_{\substack{1 \leq k \leq r \\ \beta_k > 0}} W_{2^{\beta_k}}^{d_k}$$ as $\mathfrak{w}_q$-modules.\end{lause}

As a corollary of Theorem \ref{thm:mainthmextnilpotent} and Theorem \ref{thm:mainthmsymnilpotent}, we also get reciprocity theorems for the decomposition of $S^2(W_n)$ and $\wedge^2(W_n)$, analogously to Theorem \ref{thm:unipext} and Theorem \ref{thm:unipsym} above. The proofs will be given in Section \ref{section:jordanbasis2}.

\begin{lause}\label{thm:mainthmrecursivenilpotentext}
Suppose that $q/2 < n \leq q$. Then we have $$\wedge^2(W_n) \cong \wedge^2(W_{q-n}) \oplus W_{q-1}^{n-q/2}$$ as $\mathfrak{w}_q$-modules.
\end{lause}

\begin{lause}\label{thm:mainthmrecursivenilpotentsym}
Suppose that $q/2 < n \leq q$. Then we have $$S^2(W_n) \cong S^2(W_{q-n}) \oplus W_q^{n-q/2} \oplus W_1^{n-q/2}$$ as $\mathfrak{w}_q$-modules.
\end{lause}

\begin{remark}As a corollary of Theorems \ref{thm:unipext} - \ref{thm:unipsym}, one can also give explicit expressions for the decompositions of $\wedge^2(V_n)$ and $S^2(V_n)$ in terms of the consecutive-ones binary expansion of $n$. We omit the details, but the main observation to make is that if $q = 2^{\beta_1}$ is the first term in the consecutive-ones binary expansion of $n$, then $q/2 < n \leq q$.\end{remark}




To illustrate Theorems \ref{thm:unipext} -- \ref{thm:unipsym} and \ref{thm:mainthmextnilpotent} -- \ref{thm:mainthmrecursivenilpotentsym}, some examples are provided in Table \ref{table:examples}.

\begin{table}[!htbp]
\centering
\caption{Exterior and symmetric squares of $V_n$ and $W_n$.}\label{table:examples}
\begin{tabular}{| c | l | l | l | l |}
\hline
  $n$ & $\wedge^2(V_n)$ & $S^2(V_n)$ & $\wedge^2(W_n)$ & $S^2(W_n)$ \\
	\hline
  $1$ & $0$                               & $V_1$                            & $0$                     & $W_1$ \\
	$2$ & $V_1$                             & $V_1 \oplus V_2$                 & $W_1$                   & $W_1 \oplus W_2$ \\
  $3$ & $V_3$                             & $V_2 \oplus V_4$                 & $W_3$                   & $W_1^2 \oplus W_4$ \\
	$4$ & $V_2 \oplus V_4$                  & $V_2 \oplus V_4^2$               & $W_3^2$                 & $W_1^2 \oplus W_4^2$ \\
	$5$ & $V_3 \oplus V_7$                  & $V_3 \oplus V_4 \oplus V_8$      & $W_3 \oplus W_7$        & $W_1^3 \oplus W_4 \oplus W_8$ \\
	$6$ & $V_1 \oplus V_6 \oplus V_8$       & $V_1 \oplus V_4 \oplus V_8^2$    & $W_1 \oplus W_7^2$      & $W_1^3 \oplus W_2 \oplus W_8^2$ \\
	$7$ & $V_5 \oplus V_8^2$                & $V_4 \oplus V_8^3$               & $W_7^3$                 & $W_1^4 \oplus W_8^3$ \\
	$8$ & $V_4 \oplus V_8^3$                & $V_4 \oplus V_8^4$               & $W_7^4$                 & $W_1^4 \oplus W_8^4$ \\
	$9$ & $V_5 \oplus V_8^2 \oplus V_{15}$  & $V_5 \oplus V_8^3 \oplus V_{16}$ & $W_7^3 \oplus W_{15}$ & $W_1^5 \oplus W_8^3 \oplus W_{16}$ \\
\hline
\end{tabular}

\end{table}

\section{Decomposition of $S^2(V_n)$}\label{section:unipsym}

In this section, we will prove Theorem \ref{thm:unipsym}, which gives a recursive description for the decomposition of $S^2(V_n)$ into indecomposable summands. As mentioned in the introduction, the proof follows essentially the same steps as the proof of Theorem \ref{thm:unipext} in \cite{GowLaffey}.

Let $G$ be a cyclic $2$-group of order $q > 1$ with generator $g$, and let $H = \langle g^2 \rangle$ be the unique subgroup of index $2$ in $G$. As in the introduction, we set $V_0 = 0$ and denote the indecomposable $K[G]$-modules by $V_1$, $\ldots$, $V_q$. Similarly we will set $U_0 = 0$ and denote the indecomposable $K[H]$-modules by $U_1$, $\ldots$, $U_{q/2}$, where $\dim U_i = i$ for all $1 \leq i \leq q/2$.

The restriction of a $K[G]$-module $V$ to $H$ will be denoted by $V_H$. For a $K[H]$-module $U$, we denote the induced module of $U$ from $H$ to $G$ by $U^G := K[G] \otimes_{K[H]} U$. A basic fact we will use in this section without mention is that $$U_s^G \cong V_{2s}$$ for all $1 \leq s \leq q/2$. This follows either by a direct calculation or by Green's indecomposability theorem \cite[Theorem 8]{GreenIndecomposables}.

We will denote the tensor induced module of $U$ from $H$ to $G$ by $U^{\otimes G}$ \cite[{\S 13}]{CurtisReinerVol1}. In our setting, we have $U^{\otimes G} = U \otimes U$ as a $K[H]$-module, and the action of $G$ on $U^{\otimes G}$ is defined by $$g \cdot (v \otimes w) = g^2 w \otimes v$$ for all $v, w \in U$.

We begin with a series of lemmas which are similar (or the same) as those in \cite{GowLaffey}. After this we will proceed with the proof of Theorem \ref{thm:unipsym}.

\begin{lemma}[{\cite[Lemma 5]{GowLaffey}}]\label{lemma:inducedextsquare}
Let $U$ be a $K[H]$-module. Then $$\wedge^2(U^G) \cong \wedge^2(U)^G \oplus U^{\otimes G}$$ as $K[G]$-modules.
\end{lemma}

\begin{lemma}\label{lemma:inducedsymsquare}
Let $U$ be a $K[H]$-module. Then $$S^2(U^G) \cong S^2(U)^G \oplus U^{\otimes G}$$ as $K[G]$-modules.
\end{lemma}

\begin{proof}Let $Z_1$ be the subspace of $S^2(U^G)$ spanned by $(1 \otimes v)(1 \otimes w)$ and $(g \otimes v)(g \otimes w)$ for $v, w \in U$, and let $Z_2$ be the subspace spanned by $(1 \otimes v)(g \otimes w)$ for $v, w \in U$. We have $S^2(U^G) = Z_1 \oplus Z_2$. Arguing as in \cite[proof of Lemma 5]{GowLaffey}, we see that $Z_1$ and $Z_2$ are $G$-submodules, with $Z_1 \cong S^2(U)^G$ and $Z_2 \cong U^{\otimes G}$.\end{proof}


\begin{lemma}[{\cite[Lemma 6]{GowLaffey}}]\label{lemma:uniquerestriction}
Let $V$ be a $K[G]$-module with $V_H \cong \bigoplus_{1 \leq j \leq q/2} U_j^{r_j}$. Suppose that for all odd $1 \leq j < q/2$, we have $r_j \in \{0,1\}$. Then the isomorphism type of $V$ is uniquely determined by $V_H$. 
\end{lemma}

\begin{lemma}\label{lemma:unips2multipl}
Let $n > 1$ and suppose that Theorem \ref{thm:unipsym} holds for $n$. Write $S^2(V_n) \cong \bigoplus_{j \geq 1} V_j^{r_j}$, where $r_j \geq 0$. Then the following statements hold:
\begin{enumerate}[\normalfont (i)]
\item $r_1 = 1$ if $n \equiv 2 \mod{4}$, and $r_1 = 0$ otherwise.
\item Let $j > 1$ be an odd integer. Then $r_j = 0$ if $n$ is even, and $r_j \in \{0,1\}$ if $n$ is odd.
\end{enumerate} 
\end{lemma}

\begin{proof}For $n = 2$ we have $S^2(V_n) \cong V_2 \oplus V_1$ and clearly the claim holds. Suppose then that $n > 2$. We have $S^2(V_n) \cong \wedge^2(V_{q-n}) \oplus V_{q}^{n-q/2} \oplus V_{q/2}$ as $K[G]$-modules, so the claim is immediate from the observation in \cite[p. 670]{GowLaffey}.\end{proof}

\begin{proof}[Proof of Theorem \ref{thm:unipsym}] By induction on $n$. For the base case $n = 2$ an easy calculation shows that $S^2(V_n) \cong V_2 \oplus V_1$, so the claim holds. Suppose then that $n > 2$. As in \cite[Proof of Theorem 2]{GowLaffey}, we split the proof into two cases.\\

\noindent \emph{Case 1: $n$ is even.} Write $n = 2s$. We have $V_n \cong U_s^G$, so by Lemma \ref{lemma:inducedsymsquare}\begin{equation}\label{eq:uniproof1}S^2(V_n) \cong S^2(U_s)^G \oplus U_s^{\otimes G}.\end{equation} By the induction assumption, we have $S^2(U_s) \cong \wedge^2(U_{q/2-s}) \oplus U_{q/2}^{s-q/2} \oplus U_{q/4}$ as $K[H]$-modules, and by \cite[Corollary 4]{GowLaffey} we have $U_s^{\otimes G} \cong U_{q/2 - s}^{\otimes G} \oplus V_q^{s-q/4}$ as $K[G]$-modules. Plugging these isomorphisms into~\eqref{eq:uniproof1}, we get $$S^2(V_n) \cong \wedge^2(U_{q/2-s})^G \oplus U_s^{\otimes G} \oplus V_{q}^{n-q/2} \oplus V_{q/2}$$ as $K[G]$-modules. Thus $S^2(V_n) \cong \wedge^2(V_{q-n}) \oplus V_{q}^{n-q/2} \oplus V_{q/2}$ by Lemma \ref{lemma:inducedextsquare}.\\

\noindent \emph{Case 2: $n$ is odd.} Write $n = 2s+1$. We have $(V_n)_H = U_s \oplus U_{s+1}$, so by~\eqref{eq:additivesym} $$S^2(V_n)_H \cong S^2(U_s) \oplus S^2(U_{s+1}) \oplus (U_s \otimes U_{s+1})$$ as $K[H]$-modules. Exactly one of $s$ and $s+1$ is odd, so by Lemma \ref{lemma:unips2multipl} and \cite[Corollary 2]{GowLaffey} we conclude that for all odd $1 \leq j < q/2$, the multiplicity of $U_j$ in $S^2(V_n)_H$ is either $0$ or $1$. Thus Lemma \ref{lemma:uniquerestriction} applies and $S^2(V_n)$ is uniquely determined up to isomorphism by the restriction $S^2(V_n)_H$. 

Applying the induction assumption and Theorem \ref{thm:tensordecompchar2} (ii), we get \begin{align*}S^2(U_s) &\cong \wedge^2(U_{q/2-s}) \oplus U_{q/2}^{s-q/4} \oplus U_{q/4} \\
S^2(U_{s+1}) &\cong \wedge^2(U_{q/2-s-1}) \oplus U_{q/2}^{s+1-q/4} \oplus U_{q/4} \\
U_s \otimes U_{s+1} &\cong (U_{q/2-s-1} \otimes U_{q/2-s}) \oplus U_{q/2}^{n-q/2}\end{align*} as $K[H]$-modules. Hence $S^2(V_n)_H \cong \wedge^2(U_{q/2-s} \oplus U_{q/2-s-1}) \oplus U_{q/2}^{2n-q} \oplus U_{q/4}^2,$ and so $S^2(V_n)$ has the same restriction to $H$ as $\wedge^2(V_{q-n}) \oplus V_{q}^{n-q/2} \oplus V_{q/2}$. Thus $S^2(V_n) \cong \wedge^2(V_{q-n}) \oplus V_{q}^{n-q/2} \oplus V_{q/2}$ as $K[G]$-modules.\end{proof}

\section{A Jordan basis for $W_n \otimes W_n$}\label{section:jordanbasis}

For this section, fix an integer $n > 0$, and let $q > 0$ be a power of $2$ such that $q \geq n$. Recall that for a nilpotent linear map $e: V \rightarrow V$, a \emph{Jordan chain} is a set of non-zero vectors $\{w, ew, \ldots, e^kw\}$, where $k \geq 0$ and $e^{k+1}w = 0$. A \emph{Jordan basis} for the action of $e$ on $V$ is a basis of $V$ which is a disjoint union of such Jordan chains. We denote $V^e := \{ v \in V : ev = 0 \}$

In this section, we give an explicit description of the indecomposable summands of $W_n \otimes W_n$, in terms of the consecutive-ones binary expansion of $n$. This is essentially due to Glasby, Praeger, and Xia \cite[Theorem 15]{GlasbyPraegerXiapart} --- see Proposition \ref{prop:uninilconnection}. We give a different proof by constructing a Jordan basis for the action of a generator $e$ of $\mathfrak{w}_q$ on $W_n \otimes W_n$ (Theorem \ref{thm:basistensorsquare}). 

Our construction of the Jordan basis is based on the following elementary lemma concerning Jordan chains of nilpotent linear maps. 

\begin{lemma}\label{lemma:jordanchains}
Let $e: V \rightarrow V$ be a nilpotent linear map. Let $\{ z_1, \ldots, z_t \}$ be a set of linearly independent vectors from $V^e$. Let $k_1, \ldots, k_t \geq 0$ be integers and $w_1, \ldots, w_t \in V$ such that $e^{k_i} w_i = z_i$ for all $1 \leq i \leq t$. Then $$\{ e^j w_i : 1 \leq i \leq t \text{ and } 0 \leq j \leq k_i \}$$ is a set of linearly independent vectors.
\end{lemma}

\begin{proof}By induction on $\dim V$. There is nothing to prove when $\dim V = 0$, since in this case $t = 0$. Suppose then that $\dim V > 0$. 

Since $e^{k_i} w_i = z_i$ for all $1 \leq i \leq t$, the image of $\{ e^{k_i - 1} w_i : 1 \leq i \leq t \text{ and } k_i > 0\}$ in $V/V^e$ is linearly independent and lies in $(V/V^e)^e$. Thus by applying induction on $V/V^e$, it follows that the image of $S = \{ e^j w_i : 1 \leq i \leq t \text{ and } 0 \leq j < k_i \}$ in $V/V^e$ is linearly independent. From this we conclude that $$S \cup \{z_1, \ldots, z_t \} = \{ e^j w_i : 1 \leq i \leq t \text{ and } 0 \leq j \leq k_i \}$$ is a set of linearly independent vectors.\end{proof} 

For all that follows, we fix a generator $e$ of $\mathfrak{w}_q$ and let $v_1$, $\ldots$, $v_n$ be a basis of $W_n$ such that $ev_1 = 0$ and $ev_i = v_{i-1}$ for all $1 < i \leq n$. For convenience of notation, we define $v_j = 0$ for all $j \leq 0$ and $j > n$. 

The action of $e$ on $W_n \otimes W_n$ is given by $f \otimes \operatorname{id} + \operatorname{id} \otimes f$, where $f$ is the action of $e$ on $W_n$. Thus an application of the binomial theorem shows that for all integers $i,j \leq n$ and $k \geq 0$, we have \begin{equation}\label{eq:binomformula}e^k \cdot (v_i \otimes v_j) = \sum_{0 \leq t \leq k} \binom{k}{t} v_{i-t} \otimes v_{j-k+t}.\end{equation}

For $1 \leq s \leq n$, define $$z_s := \sum_{1 \leq i \leq s} v_i \otimes v_{s+1-i}.$$ It is clear that $e \cdot z_s = 0$ for all $1 \leq s \leq n$, and in fact we have the following.

\begin{lemma}\label{lemma:fixedpointstensor}
The set $\{ z_1, \ldots, z_n \}$ is a basis of $(W_n \otimes W_n)^e$.
\end{lemma}

\begin{proof}A straightforward calculation --- see for example \cite[Lemma 2]{Norman}.\end{proof}																																																														
																																																																	
Let $n = \sum_{1 \leq i \leq r} (-1)^{i+1} 2^{\beta_i}$ be the consecutive-ones binary expansion of $n$, where $\beta_1 > \cdots > \beta_r \geq 0$, and $\beta_{r-1} > \beta_r + 1$ if $r > 1$. Define $$n_k := \sum_{k \leq i \leq r} (-1)^{i+k} 2^{\beta_i}$$ for all $1 \leq k \leq r$, and set $n_{r+1} := 0$. Note that $n = n_1 > n_2 > \cdots > n_r > n_{r+1} = 0$. 

The rest of this section proceeds as follows. Consider $1 \leq s \leq n$ and let $1 \leq k \leq r$ be the unique integer such that $n_k > n-s \geq n_{k+1}$. Using the next two lemmas, we will construct $w_s \in W_n \otimes W_n$ such that $e^{2^{\beta_k}-1} w_s = z_s$. From this, an application of Lemma \ref{lemma:jordanchains} will give us a Jordan basis for the action of $e$ on $W_n \otimes W_n$.

\begin{lemma}\label{lemma:mod2b1calc}
Let $\beta = \beta_k > 0$ and $1 \leq s \leq n$ with $n-s \geq n_{k+1}$. Then there exists an integer $s \leq x \leq 2n-s$ such that $x \equiv 2^{\beta} \mod{2^{\beta+1}}$.
\end{lemma}

\begin{proof}First note that the claim holds if $n-s \geq 2^{\beta}$, since in this case the interval $[s, 2n-s]$ contains a complete set of representatives modulo $2^{\beta+1}$. This fact will be used throughout the proof, which we split into two cases:\\

\noindent \emph{Case 1: $k \equiv 0 \mod{2}$.} In this case $n = n_{k+1} + 2^{\beta} + n'2^{\beta+1}$ for some $n' \geq 0$. 

Suppose first that $2^\beta$ occurs in the binary expansion of $s$, so $s = s'' + 2^{\beta} + s'2^{\beta+1}$ for some $0 \leq s' \leq n'$ and $0 \leq s'' < 2^{\beta}$. If $n' > s'$, then $n-s > 2^{\beta}$ and the claim holds. If $n' = s'$, then $n-s = n_{k+1} - s''$, so $s'' = 0$ since $n-s \geq n_{k+1}$. Thus we can choose $x = s \equiv 2^{\beta} \mod{2^{\beta+1}}$.

If $2^{\beta}$ does not occur in the binary expansion of $s$, then $s = s'' + s'2^{\beta+1}$ for some $0 \leq s' \leq n'$ and $0 \leq s'' < 2^{\beta}$. We have $n-s = n_{k+1}-s'' + 2^{\beta} \geq 2^\beta - s''$. Thus we can choose $x = s + (2^{\beta}-s'') \equiv 2^{\beta} \mod{2^{\beta+1}}$.\\

\noindent \emph{Case 2: $k \not\equiv 0 \mod{2}$.} We have $n = -n_{k+1} + 2^{\beta} + n'2^{\beta+1}$ for some $n' \geq 0$. 

As in the previous case, suppose first that $2^\beta$ occurs in the binary expansion of $s$. Then $s = s'' + 2^{\beta} + s'2^{\beta+1}$ for some $0 \leq s' < n'$ and $0 \leq s'' < 2^{\beta}$. If $n' > s' + 1$, then $n-s > 2^{\beta}$ and the claim holds. If $n' = s'+1$, we have $$n-s = 2^{\beta+1}-n_{k+1}-s'' \geq 2^{\beta+1}-s''-2^{\beta-1}$$ since $n_{k+1} \leq 2^{\beta-1}$. It follows then that $n-s \geq 2^{\beta-1}$, so $$2(n-s) \geq (2^{\beta+1} - s'' - 2^{\beta-1}) + 2^{\beta-1} = 2^{\beta+1}-s''.$$ Hence we can pick $x = s + (2^{\beta+1}-s'') \equiv 2^{\beta} \mod{2^{\beta+1}}$.

Consider then the case where $2^{\beta}$ does not occur in the binary expansion of $s$, so $s = s'' + s'2^{\beta+1}$ for some $0 \leq s' \leq n'$ and $0 \leq s'' < 2^{\beta}$. If $n' > s'$, we have $n-s > 2^{\beta}$, so assume that $n' = s'$. In this case $n-s = 2^{\beta}-s''-n_{k+1}$. Since $n-s \geq n_{k+1}$, it follows that $2^{\beta}-s'' \geq 2n_{k+1}$. Thus $$2(n-s) = 2(2^{\beta}-s'') - 2n_{k+1} \geq 2^{\beta}-s'',$$ so we can choose $x = s + (2^{\beta}-s'') \equiv 2^{\beta} \mod{2^{\beta+1}}$.
\end{proof}

\begin{lemma}\label{lemma:j0lemma}
Let $\beta = \beta_k > 0$ and $1 \leq s \leq n$ with $n-s \geq n_{k+1}$. Then there exists an integer $j_0 \geq 0$ such that the following hold: 

\begin{enumerate}[\normalfont (i)]
\item $s \leq \lfloor s/2 \rfloor + 2^{\beta-1} + j_02^{\beta} \leq n$,
\item $s \leq \lceil s/2 \rceil + 2^{\beta-1} + j_02^{\beta} \leq n$.
\end{enumerate}
\end{lemma}

\begin{proof}If $s$ is even, both (i) and (ii) are equivalent to $s \leq 2^{\beta} + j_02^{\beta+1} \leq 2n-s$, so the existence of such a $j_0 \geq 0$ follows from Lemma \ref{lemma:mod2b1calc}. If $s$ is odd, then both (i) and (ii) hold if and only if $s+1 \leq 2^{\beta} + j_02^{\beta+1} \leq 2n-s-1$. In this case, the existence of such a $j_0 \geq 0$ follows from Lemma \ref{lemma:mod2b1calc} since $s$ and $2n-s$ are odd.\end{proof}

For the next lemma, we fix $1 \leq k \leq r$ and set $\beta := \beta_k$. For $1 \leq s \leq n$ with $n_k > n-s \geq n_{k+1}$, we define a vector $w_s \in W_n \otimes W_n$ as follows. If $\beta = 0$, we set $w_s = z_s$. If $\beta > 0$, we define $$w_s := \sum_{-j_0 \leq j \leq j_0} v_{\lfloor s/2 \rfloor + 2^{\beta-1} + j2^{\beta}} \otimes v_{\lceil s/2 \rceil + 2^{\beta-1} - j2^{\beta}},$$ where $j_0 \geq 0$ is as in Lemma \ref{lemma:j0lemma}.

\begin{lemma}\label{lemma:2b1lemma}
Let $1 \leq s \leq n$ with $n_k > n-s \geq n_{k+1}$. Then $e^{2^{\beta}-1} w_s = z_s$.
\end{lemma}

\begin{proof}If $\beta = 0$, there is nothing to prove since $w_s = z_s$. Suppose then that $\beta > 0$. By Lucas' theorem, we have $\binom{2^{\beta}-1}{t} \equiv 1 \mod{2}$ for all $0 \leq t \leq 2^{\beta}-1$. Thus with~\eqref{eq:binomformula}, we get \begin{align}\nonumber e^{2^{\beta}-1} \cdot (v_i \otimes v_j) &= \sum_{0 \leq t \leq 2^{\beta}-1} v_{i-t} \otimes v_{j-2^{\beta}+1+t} \\ \label{eq:bb2} &= \sum_{i-2^{\beta}+1 \leq t \leq i} v_t \otimes v_{i+j-2^{\beta}+1-t} \end{align} for all integers $i,j \leq n$. By Lemma \ref{lemma:j0lemma} each summand in the definition of $w_s$ is of the form $v_i \otimes v_j$ for some $i,j \leq n$, so by~\eqref{eq:bb2} we have $$e^{2^{\beta}-1} \cdot w_s = \sum_{\ell \leq t \leq \ell'} v_t \otimes v_{s+1-t},$$ where $\ell = \lfloor s/2 \rfloor + 2^{\beta-1} - (j_0+1)2^{\beta} + 1$ and $\ell' = \lfloor s/2 \rfloor + 2^{\beta-1} + j_02^{\beta}$.

Thus in order to prove that $e^{2^{\beta}-1} \cdot w_s = z_s$, it will suffice to show that $\ell \leq 1$ and $\ell' \geq s$. First note that the inequality $\ell' \geq s$ is just Lemma \ref{lemma:j0lemma} (i). Next, by Lemma \ref{lemma:j0lemma} (ii), we have $\lfloor s/2 \rfloor = s - \lceil s/2 \rceil \leq 2^{\beta-1} + j_02^{\beta}.$ Thus $$\ell \leq (2^{\beta-1} + j_02^{\beta}) + 2^{\beta-1} - (j_0+1)2^{\beta} + 1 = 1,$$ which completes the proof of the lemma.\end{proof}

We are now ready to prove the main result of this section.

\begin{lause}\label{thm:basistensorsquare}
For $1 \leq k \leq r$, define $$B_k := \{ e^j w_s : n_k > n-s \geq n_{k+1} \text { and } 0 \leq j \leq 2^{\beta_k}-1 \}.$$ Then $B := \cup_{1 \leq k \leq r} B_k$ is a Jordan basis for the action of $e$ on $W_n \otimes W_n$.
\end{lause}

\begin{proof}By Lemma \ref{lemma:2b1lemma} and Lemma \ref{lemma:jordanchains}, the vectors in $B$ are linearly independent. To prove that $B$ is a Jordan basis, it will suffice to show that $|B| = \dim W_n \otimes W_n = n^2$. To this end, note first that $|B_k| = 2^{\beta_k}(n_k - n_{k+1})$ for all $1 \leq k \leq r$. Furthermore, we have \begin{equation}\label{eq:nkdif}n_k - n_{k+1} = 2^{\beta_k} + 2 \sum_{k < i \leq r} (-1)^{i+k} 2^{\beta_i}\end{equation} for all $1 \leq k \leq r$. With~\eqref{eq:nkdif}, a straightforward calculation shows that \begin{equation}\label{eq:n2equality}n^2 = \sum_{1 \leq k \leq r} 2^{\beta_k} (n_k - n_{k+1}).\end{equation} Hence $n^2 = \sum_{1 \leq k \leq r} |B_k| = |B|$, which completes the proof.\end{proof}

As an immediate corollary, we recover the following result from \cite{GlasbyPraegerXiapart} --- see Proposition \ref{prop:uninilconnection}.

\begin{lause}[{\cite[Theorem 15]{GlasbyPraegerXiapart}}] For $1 \leq k \leq r$, define $d_k := 2^{\beta_k} + \sum_{k < i \leq r} (-1)^{k+i} 2^{\beta_i+1}$. Then $$W_n \otimes W_n \cong \bigoplus_{1 \leq k \leq r} W_{2^{\beta_k}}^{d_k}$$ as $\mathfrak{w}_q$-modules.
\end{lause}

\section{Decomposition of $S^2(W_n)$ and $\wedge^2(W_n)$}\label{section:jordanbasis2}

We continue with the notation from the previous section. In this section, we will describe a Jordan basis for the action of $e$ on $S^2(W_n)$, using the Jordan basis of $e$ on $W_n \otimes W_n$ described in Theorem \ref{thm:basistensorsquare}. This will then allow us to prove Theorem \ref{thm:mainthmextnilpotent} and Theorem \ref{thm:mainthmsymnilpotent}, our main results for $\wedge^2(W_n)$ and $S^2(W_n)$. 

Let $\pi: W_n \otimes W_n \rightarrow S^2(W_n)$ be the natural quotient map, defined by $\pi(v \otimes w) = vw$ for all $v,w \in W_n$.

\begin{lause}\label{thm:jordanbasissymsquare} Let $B_0' := \{\pi(w_s) : 1 \leq s \leq n \text { even}\}$ and $$B_k' := \{e^j \pi(w_s) : s \text{ odd, } n_k > n-s \geq n_{k+1}, \text{ and } 0 \leq j \leq 2^{\beta_k}-1\}$$ for all $1 \leq k \leq r$. Then $B' := \bigcup_{0 \leq k \leq r} B_k'$ is a Jordan basis for the action of $e$ on $S^2(W_n)$.\end{lause}

\begin{proof}Let $B$ be the Jordan basis of $W_n \otimes W_n$ as in Theorem \ref{thm:basistensorsquare}. We will begin by showing that \begin{equation}\label{eq:basisequality}B' = \pi(B) \setminus \{0\},\end{equation} which implies that $B'$ spans $S^2(W_n)$. To this end, first note that for $1 \leq s \leq n$ odd, we have $\pi(z_s) = \pi(v_{(s+1)/2} \otimes v_{(s+1)/2}) \neq 0$. Furthermore, the map $\pi$ is invariant under the action of $e$ and $e^{2^{\beta_k}} w_s = z_s$ by Lemma \ref{lemma:2b1lemma}. We conclude then that for odd $s$ with $n_k > n-s \geq n_{k+1}$, we have $\pi(e^j w_s) = e^j \pi(w_s) \neq 0$ for all $0 \leq j \leq 2^{\beta_k}-1$.

For $1 \leq s \leq n$ even, we have $\pi(w_s) = \pi(v_{s/2+2^{\beta-1}} \otimes v_{s/2+2^{\beta-1}}) \in S^2(W_n)^e$ for some $\beta > 0$, and thus $\pi(e^j w_s) = 0$ for all $j > 0$. This completes the proof of~\eqref{eq:basisequality}.

Now to show that $B'$ is a Jordan basis for the action of $e$ on $S^2(W_n)$, it will suffice to show that $|B'| = \dim S^2(W_n) = n(n+1)/2$. Consider first the case where $n$ is even. Then $$\sum_{1 \leq k \leq r} |B_k'| = \sum_{1 \leq k \leq r} 2^{\beta_k} \frac{(n_k - n_{k+1})}{2} = n^2 / 2$$ by~\eqref{eq:n2equality}. Furthermore, we have $|B_0'| = n/2$, so $|B| = n(n+1)/2$. Suppose next that $n$ is odd. In this case $\beta_r = 0$ and $|B_r| = 1$, so $$\sum_{1 \leq k \leq r} |B_k'| = 1 + \sum_{1 \leq k \leq r-1} 2^{\beta_k} \frac{(n_k - n_{k+1})}{2} = (n^2+1)/2$$ by~\eqref{eq:n2equality}. Since $|B_0'| = (n-1)/2$, we again get $|B| = n(n+1)/2$, as claimed.\end{proof}

We can now prove our main result for $S^2(W_n)$.

\begin{proof}[Proof of Theorem \ref{thm:mainthmsymnilpotent}]This is just a matter of counting Jordan chains in the basis $B'$ described in Theorem \ref{thm:jordanbasissymsquare}. The Jordan block sizes that occur are $2^{\beta_k}$ for $1 \leq k \leq r$. For $\beta_k > 0$, the multiplicity of $2^{\beta_k}$ is $|B_k'|/2^{\beta_k} = (n_k - n_{k+1})/2 = d_k$. Now what remains is to count the number of blocks of size $1$. If $n$ is even, this is given by $|B_0'| = n/2$. If $n$ is odd, we have $\beta_r = 0$ and the multiplicity is given by $|B_0'| + |B_r'| = (n+1)/2$.\end{proof} 

As a corollary of Theorem \ref{thm:jordanbasissymsquare}, we also get the following.

\begin{seur}\label{corollary:fixedptsymsq}
The action of $e$ has $n$ Jordan blocks on $S^2(W_n)$.
\end{seur}

\begin{proof}Let $B$ and $B'$ be the Jordan bases described in Theorem \ref{thm:basistensorsquare} and Theorem \ref{thm:jordanbasissymsquare}, respectively. Each Jordan chain in $B$ is mapped by $\pi$ to a Jordan chain in $B'$. Therefore $e$ has the same number of Jordan blocks on $W_n \otimes W_n$ and $S^2(W_n)$, and the claim follows from Lemma \ref{lemma:fixedpointstensor}.\end{proof}

We are now ready to prove the rest of our main results: the decomposition theorem for $\wedge^2(W_n)$ (Theorem \ref{thm:mainthmextnilpotent}) and the recurrence relations for $\wedge^2(W_n)$ and $S^2(W_n)$ (Theorems \ref{thm:mainthmrecursivenilpotentext} and \ref{thm:mainthmrecursivenilpotentsym}).

\begin{proof}[Proof of Theorem \ref{thm:mainthmextnilpotent}]Let $\pi': S^2(W_n) \rightarrow \wedge^2(W_n)$ be the map defined by $\pi'(vw) = v \wedge w$ for all $v,w \in W_n$. Then we have a short exact sequence $$0 \rightarrow W_n^{[2]} \rightarrow S^2(W_n) \xrightarrow{\pi'} \wedge^2(W_n) \rightarrow 0$$ of $\mathfrak{w}_q$-modules, where $W_n^{[2]}$ is the subspace of $S^2(W_n)$ spanned by elements of the form $v^2$ for $v \in V$. Now $\dim W_n^{[2]} = n$ and $W_n^{[2]}$ is annihilated by the action of $e$, so by Corollary \ref{corollary:fixedptsymsq} we have $W_n^{[2]} = S^2(W_n)^e$. Therefore \begin{equation}\label{eq:modfixspace}\wedge^2(W_n) \cong S^2(W_n) / S^2(W_n)^e\end{equation} as $\mathfrak{w}_q$-modules.

For any $\mathfrak{w}_q$-module $V$, it is easy to see that if $V \cong W_{r_1} \oplus \cdots \oplus W_{r_t}$ for some integers $r_1, \ldots, r_t > 0$, then $V/V^e \cong W_{r_1 - 1} \oplus \cdots \oplus W_{r_t - 1}$. Thus the claim follows from~\eqref{eq:modfixspace} and Theorem \ref{thm:mainthmsymnilpotent}.\end{proof}

\begin{proof}[Proof of Theorem \ref{thm:mainthmrecursivenilpotentext}]
With the assumption $q/2 < n \leq q$, we have $q = 2^{\beta_1}$. By Theorem \ref{thm:mainthmextnilpotent}, we have $$\wedge^2(W_{q-n}) \cong \bigoplus_{\substack{1 < k \leq r \\ \beta_k > 0}} W_{2^{\beta_k}-1}^{d_k}$$ as $\mathfrak{w}_q$-modules, where $d_k := 2^{\beta_k - 1} + \sum_{k < i \leq r} (-1)^{k+i} 2^{\beta_i}$ for all $1 \leq k \leq r$. Thus by applying Theorem \ref{thm:mainthmextnilpotent} to $\wedge^2(W_n)$, we conclude that $$\wedge^2(W_n) \cong \wedge^2(W_{q-n}) \oplus W_{2^{\beta_1}}^{d_1}$$ as $\mathfrak{w}_q$-modules. Since $d_1 = n-q/2$, the claim follows.\end{proof}

\begin{proof}[Proof of Theorem \ref{thm:mainthmrecursivenilpotentsym}]Using the same argument as in the proof of Theorem \ref{thm:mainthmrecursivenilpotentext}, the result follows from Theorem \ref{thm:mainthmsymnilpotent}.\end{proof}


\end{document}